\documentclass[12pt,twoside]{amsart}
\usepackage{amscd,amssymb,amsthm,amsmath,amsfonts,amscd,bm,enumerate,latexsym}
\usepackage[all]{xy}

\newtheorem{theo}{Theorem}[section]
\newtheorem{prop}[theo]{Proposition}
\newtheorem{lem}[theo]{Lemma}
\newtheorem*{claim}{Claim}

\newtheorem{cor}[theo]{Corollary}

\theoremstyle{definition}

\newtheorem{rem}[theo]{Remark}
\newtheorem{prob}[theo]{Problem}

 \newcommand\Coh{\mathop{\mathrm{Coh}}\nolimits}
 \newcommand\Spec{\mathop{\mathrm{Spec}}\nolimits}
 \newcommand\OO{\mathop{\mathcal{O}}\nolimits}
  \newcommand\Hom{\mathop{\mathrm{Hom}}\nolimits}
 \newcommand\Ext{\mathop{\mathrm{Ext}}\nolimits}
 
\begin{document}
 \title[]{Semi-orthogonal decomposability of the derived category of a curve}
 \author{Shinnosuke OKAWA}
 \address{Graduate School of Mathematical Sciences, 
the University of Tokyo, 3-8-1 Komaba, Meguro-ku, Tokyo 153-8914, Japan.}
\email{okawa@ms.u-tokyo.ac.jp}
\date{}

 \maketitle
 
 \begin{abstract}
 We show that 
 the bounded derived category of coherent sheaves on a smooth projective curve except
 the projective line admits no non-trivial semi-orthogonal decompositions. 
  \end{abstract}
 
\tableofcontents

\section{Introduction}

It is expected (and has been partially confirmed) that we can understand the minimal model program
(MMP for short) in terms of the semi-orthogonal decompositions (SOD for short) of the derived category of
coherent sheaves (see \cite{h} for the definition of derived categories of sheaves and SODs).
To be precise we expect that we can define a suitable triangulated category
for each projective variety with mild singularities, which equals to the bounded derived category of coherent sheaves
when the variety is smooth, and that each step of MMP yields a
non-trivial SOD of the category. In particular we expect that a variety whose derived category
admits no non-trivial SOD is minimal in the sense of MMP (see \cite{k} and \cite{bo}).

On the other hand, the derived category of a minimal model may admit a non-trivial SOD, even if the
canonical line bundle is ample. In fact, a Godeaux surface gives such an example (see Section \ref{higher dimensions}). 

In view of this, it is interesting to study what kind of positivity of the canonical line bundle guarantees the
non-existence of non-trivial SODs.

In this paper we prove the following theorem.

\begin{theo}\label{main}
Let $C$ be a smooth projective curve of genus at least one. Then
$D^{b}(C)$, the bounded derived category of coherent sheaves on $C$, 
admits no non-trivial semi-orthogonal decompositions.
\end{theo}

It is well known that the derived category of a smooth projective variety with trivial canonical line bundle
admits no non-trivial SODs,
since in such a case the Serre functor is just a shift by the dimension.
Therefore the case $g(C)\ge 2$ is essential.

This result seems to be believed by specialists, but the author could not
find it in the literature.

In view of Theorem \ref{main}, it is natural to ask if the similar results hold in higher dimensions:
\begin{prob}\label{problem}
Let $X$ be a smooth projective variety whose canonical line bundle
is globally generated. Does $D^{b}(X)$ admit a non-trivial SOD?
\end{prob}

As a partial answer, for such $X$ we check that $D^{b}(X)$ has no exceptional object (see Section \ref{higher dimensions}).
\section{Proof}\label{Proof}
In this section we give a proof for Theorem \ref{main}.

The following lemma ($=$\cite[Lemma 7.2]{gkr}) is the key ingredient of the proof :

\begin{lem}\label{lem:GKR}
Let $C$ be a smooth projective curve of genus $g(C) \geq 1$. Suppose $E \in \Coh (C)$ is included in a triangle
$$Y \to E \to X \to Y [1]$$
with $Hom^{\leq 0} (Y, X) = 0$. Then $X, Y \in \Coh (C)$.  
\end{lem}

\begin{cor}\label{key}
Let $C$ be a smooth projective curve of genus $g(C) \geq 1$, and $D^{b}(C)=\langle\mathcal{A},\mathcal{B}\rangle$
be a SOD of $D^{b}(C)$. Then for any $E\in \Coh(C)$,
there exist coherent sheaves $b\in\mathcal{B}\cap{\Coh{(C)}}$ and $a\in\mathcal{A}\cap{\Coh{(C)}}$,
and an exact sequence of sheaves
\begin{equation*}
0\to b\to E\to a\to 0.
\end{equation*}
\end{cor}

\begin{proof}[Proof of Theorem \ref{main}]
Let $D^{b}(C)=\langle\mathcal{A},\mathcal{B}\rangle$ be a SOD of $D^{b}(C)$. We show that
either $\mathcal{A}$ or $\mathcal{B}$ is trivial.

Pick a closed point $x\in X$.
By Corollary \ref{key}, there exist $b\in\mathcal{B}$, $a\in\mathcal{A}$ such that both of them are sheaves and
there exists an exact sequence
\begin{equation*}
0\to b\to \OO_{x}\to a\to 0.
\end{equation*}
Therefore either 
$\OO_{x}\in\mathcal{A}$ or $\OO_{x}\in\mathcal{B}$ should hold.
Note that only one of them holds, since $\mathcal{A}\subset\mathcal{B}^{\perp}$.
Summing up we obtain the decomposition of the set of closed points of the curve $C$
\begin{equation*}
C(\Spec{(\mathbb{C})})=C_{\mathcal{A}}\amalg C_{\mathcal{B}},
\end{equation*}
where $C_{\mathcal{A}}$ (resp. $C_{\mathcal{B}}$) denotes the set of closed points belonging to $\mathcal{A}$
(resp. $\mathcal{B}$).

Suppose that $C_{\mathcal{B}}=C(\Spec{(\mathbb{C})})$. Then $\mathcal{A}$
should be trivial, since the set of closed points forms a spanning class \cite[Proposition 3.17]{h}.
In this case we are done.

Suppose that at least one closed point belongs to $\mathcal{A}$. Then we see the following
\begin{claim}
Any coherent sheaf in $\mathcal{B}$ must be torsion.
\end{claim}
\begin{proof}
Otherwise the support of the sheaf coincides with the whole variety $C$, hence there exists a
non-trivial morphism from the sheaf to a closed point which belongs to $\mathcal{A}$.
This is a contradiction.
\end{proof}

Now we show that any torsion free sheaf belongs to $\mathcal{A}$. In fact,
let $E$ be a torsion free sheaf. As before, we have an exact sequence
\begin{equation*}
0\to b\to E\to a\to 0.
\end{equation*}
Since $E$ is torsion free, so is $b$. Combined with the claim, we see $b$ must be zero, hence
$a=E$.

The set of torsion free sheaves forms a spanning class \cite[Corollary 3.19]{h},
hence $\mathcal{B}$ has to be trivial.
This concludes the proof.

\end{proof}

\begin{rem}
\begin{enumerate}
\item
Corollary \ref{key} is the only place where we used the assumption that $C$ is a smooth projective curve of
genus at least one. If $X$ is a smooth projective variety satisfying the conclusion of Corollary \ref{key},
the same argument as above holds: i.e. in such a case $D^{b}(X)$ has no non-trivial SODs.

\item
As is well known, $D^{b}(\mathbb{P}^1)$ has a non-trivial SOD \cite[Corollary 8.29]{h}.
In fact we can check that Lemma \ref{lem:GKR} does not hold for $\mathbb{P}^1$.

Set $C=\mathbb{P}^1$, and fix a closed point $p\in C$. From the exact sequence
\begin{equation*}
0\to\OO_C(-p)\to\OO_C\to\OO_{p}\to 0,
\end{equation*}
we obtain the following triangle
\begin{equation*}
\OO_C\to\OO_{p}\to\OO_C(-p)[1]\to\OO_C[1].
\end{equation*}
It is easy to check that $\Ext^{i} (\OO_C, \OO_C(-p)[1]) = 0$ holds for all $i\in\mathbb{Z}$. Thus
we see that the conclusion of Lemma \ref{lem:GKR} does not hold for $\mathbb{P}^{1}$.
\end{enumerate}
\end{rem}

\section{Higher dimensions}\label{higher dimensions}
In this section we discuss the higher dimensional case.

First of all, we point out that the ampleness of the canonical line bundle does not guarantee the
indecomposability of the derived category.
In fact, let $X$ be a Godeaux surface (see \cite[X.3(4)]{b}).
It is a smooth projective surface with ample canonical line bundle and satisfies $h^1(\OO_X)
=h^2(\OO_X)=0$.

Therefore every line bundle on $X$ is an exceptional object, hence the triangulated subcategory
generated by one of them gives a non-trivial
SOD of $D^b(X)$ (see \cite[Chapter 1 Section 4]{h}).

Finally, as a partial answer to Problem \ref{problem} we check that the global generation
of the canonical line bundle implies the non-existence of exceptional objects.

\begin{prop}\label{no exceptional object}
Under the same assumption as in Problem \ref{problem}, $D^{b}(X)$ has no exceptional object.
\end{prop}
\begin{proof}
For any non-zero object $F$ in $D^b(X)$,
\begin{equation*}
\Hom(F,F[\dim{X}])^{\vee}\cong\Hom(F,F\otimes K_X)\not=0
\end{equation*}
holds due to the following
lemma\footnote{The author learned the lemma, together with its proof, from Dr. Kotaro Kawatani.}:
\begin{lem}\label{Kawatani}
Let $F$ be a non-trivial object in $D^{b}(X)$, and $L$ be a globally generated line bundle. Then
\begin{equation*}
\Hom(F,F\otimes L)\not=0
\end{equation*}
holds.
\end{lem}
\begin{proof}
When $F$ is a shift of a sheaf, we can find a section $s\in H^0(X,L)$ such that
$\otimes s:F\to F\otimes L$ is a non-trivial homomorphism.

In general, set $m=\min\{i \ |H^i(F)\not= 0\}$ and consider the following standard triangle
\begin{equation*}
\tau_{\le m}(F)\to F\to \tau_{\ge m+1}(F)\to\tau_{\le m}(F)[1],
\end{equation*}
where $\tau_{\le m}$ (resp. $\tau_{\ge m+1}$) denotes the lower truncation of $F$ at degree $m$
(resp. upper truncation at degree $m+1$).

Since  $\tau_{\le m}(F)$ is isomorphic to a shift of a sheaf, we can find $s\in H^0(X,L)$ as above
for $\tau_{\le m}(F)$.
Consider the following morphism of triangles.

\[
\xymatrix{
\tau_{\ge m+1}(F)[-1]\ar[r]\ar[d]_{\sigma_{\ge m+1}[-1]} & \tau_{\le m}(F)\ar[r] \ar[d]_{\sigma_{\le m}}
&F\ar[r] \ar[d]_{\sigma} &\tau_{\ge m+1}(F)\ar[d]_{\sigma_{\ge m+1}}  \\
\tau_{\ge m+1}(F)\otimes L[-1]\ar[r] & \tau_{\le m}(F)\otimes L\ar[r] &F\otimes L\ar[r]&\tau_{\ge m+1}(F)\otimes L \\}\]

In the diagram above, four vertical arrows are defined by taking tensor products with the section $s$, so that
the morphism $\sigma_{\le m}$ is non-trivial.

Suppose that $\sigma=0$. Then $\sigma_{\le m}$ factors through a morphism from
$\tau_{\le m}(F)$ to $\tau_{\ge m+1}(F)\otimes L[-1]$, which is zero since $\tau_{\ge m+1}(F)\otimes L[-1]$
has trivial cohomologies up to degree $m+1$. Thus we obtain a contradiction, which means the non-triviality
of $\sigma$.
\end{proof}
\end{proof}
\section*{Acknowledgements}
The author would like to thank his advisor Professor Yurjio Kawamata for asking him
about higher dimensional cases and Dr. Kotaro Kawatani for useful discussions,
informing him of Lemma \ref{Kawatani} with its proof and kindly allowing him to include them in this paper.
He would also like to thank Professor Hokuto Uehara
for useful discussions, and Professors Alexei Bondal and Yukinobu Toda
for answering his question.

The author is supported by the Grant-in-Aid for Scientific Research
(KAKENHI No. 22-849) and the Grant-in-Aid for JSPS fellows.


 \end{document}